\documentclass[reqno,11pt]{amsart}
\usepackage{graphicx,amsfonts,amssymb,amsmath,amsthm,amscd}
\usepackage[all]{xy}
\usepackage{color}
\usepackage{hyperref}
\hypersetup{colorlinks,linkcolor=magenta,citecolor=blue,filecolor=blue,urlcolor=blue}

\theoremstyle{plain} 
\newcounter{statement}[section]

\newtheorem{theorem}    [statement]{Theorem}

\newtheorem{lemma}      [statement]{Lemma}
\newtheorem{cor}  [statement]{Corollary}
\newtheorem{prop}[statement]{Proposition}

\newcounter{intro}

\newtheorem{intro_theorem}[intro]{Theorem}
\newtheorem{intro_cor}[intro]{Corollary}

\theoremstyle{definition}
\newtheorem{definition} [statement]{Definition}
\newtheorem*{definition*}           {Definition}

\newtheorem{example}    [statement]{Example}

\theoremstyle{remark}
\newtheorem{remark}[statement]{Remark}

\newtheorem*{claim*}               {Claim}

\definecolor{prpl}{rgb}{0.7, 0.0, 0.7}
\newenvironment{JT}{\noindent \color{prpl}{\bf JT:} \footnotesize}{} \newenvironment{BB}{\noindent \color{blue}{\bf BB:} \footnotesize}{}

\def\PU{\operatorname{PU}}
\def\U{\operatorname{U}}

\def\Hom{\operatorname{Hom}}

\def\Imag{\operatorname{Im}}

\def\mult{\operatorname{mult}}
\def\vol{\operatorname{vol}}

\def\chone{\operatorname{c_1}}
\def\chtwo{\operatorname{c_2}}
\def\Ric{\operatorname{Ric}}
\def\Supp{\operatorname{Supp}}

\def\ord{\operatorname{ord}}

\def\P{\mathbb{P}}
\def\Q{\mathbb{Q}}
\def\Z{\mathbb{Z}}
\def\C{\mathbb{C}}
\def\R{\mathbb{R}}

\def\H{\mathbb{H}}

\def\B{\mathbb{B}}
\def\SS{\mathbb{S}}

\makeatletter
\newcommand{\xleftrightarrow}[2][]{\ext@arrow 3359\leftrightarrowfill@{#1}{#2}}
\newcommand{\xdashrightarrow}[2][]{\ext@arrow 0359\rightarrowfill@@{#1}{#2}}
\newcommand{\xdashleftarrow}[2][]{\ext@arrow 3095\leftarrowfill@@{#1}{#2}}
\newcommand{\xdashleftrightarrow}[2][]{\ext@arrow 3359\leftrightarrowfill@@{#1}{#2}}
\def\rightarrowfill@@{\arrowfill@@\relax\relbar\rightarrow}
\def\leftarrowfill@@{\arrowfill@@\leftarrow\relbar\relax}
\def\leftrightarrowfill@@{\arrowfill@@\leftarrow\relbar\rightarrow}
\def\arrowfill@@#1#2#3#4{%
  $\m@th\thickmuskip0mu\medmuskip\thickmuskip\thinmuskip\thickmuskip
   \relax#4#1
   \xleaders\hbox{$#4#2$}\hfill
   #3$%
}
\makeatother
\newcommand\dashto{\mathrel{
  -\mkern-6mu{\to}\mkern-20mu{\color{white}\bullet}\mkern12mu
}}

\renewcommand{\O}{\mathcal{O}}

\renewcommand{\phi}{\varphi}
\renewcommand{\tilde}[1]{\widetilde{#1}}
\renewcommand{\bar}[1]{\overline{#1}}

\def\into{\rightarrow}

 \title[The Kodaira dimension of hyperbolic manifolds]{The Kodaira dimension of complex hyperbolic manifolds with cusps}
 \date{\today}
 \author{Benjamin Bakker}
 \address{B. Bakker:
 Institut f\"ur Mathematik, Humboldt-Universit\"at zu Berlin.
 }
 \email{benjamin.bakker@math.hu-berlin.de}
\author{Jacob Tsimerman}
\address{J. Tsimerman:
Department of Mathematics, University of Toronto.
}
\email{jacobt@math.toronto.edu}

\begin{document}

\begin{abstract}  We prove a bound relating the volume of a curve near a cusp in a hyperbolic manifold to its multiplicity at the cusp.  The proof uses a hybrid technique employing both the geometry of the uniformizing group and the algebraic geometry of the toroidal compactification.  There are a number of consequences:  we show that for an $n$-dimensional toroidal compactification $\bar X$ with boundary $D$, $K_{\bar X}+(1-\frac{n+1}{2\pi}) D$ is nef, and in particular that $K_{\bar X}$ is ample for $n\geq 6$.  By an independent algebraic argument, we prove that every hyperbolic manifold of dimension $n\geq 3$ is of general type, and conclude that the phenomena famously exhibited by Hirzebruch in dimension 2 do not occur in higher dimensions.  Finally, we investigate the applications to the problem of bounding the number of cusps and to the Green--Griffiths conjecture.  
\end{abstract}

\maketitle
\section{Introduction}

Complex hyperbolic manifolds are complex manifolds admitting a complete finite-volume metric of constant negative sectional curvature.  Such manifolds are quotients of the complex hyperbolic ball by a discrete group of holomorphic isometries.  Just as for real hyperbolic manifolds, the topology of the uniformizing group is a powerful tool in studying their geometry.  On the other hand, work of \cite{amrt} and \cite{mok} shows that such manifolds always admit orbifold toroidal compactifications whose algebraic geometry provides an equally powerful complementary set of techniques.  Quotients by arithmetic lattices naturally arise as Shimura varieties parametrizing abelian varieties with certain endomorphism structure.  Interestingly, the hyperbolic ball is the only bounded symmetric domain that admits nonarithmetic lattices \cite{marguilis}, and examples have only been constructed in dimensions 2 and 3 by Mostow \cite{mostow} and Deligne--Mostow \cite{delignemostow}.

In this paper we study curves in non-compact complex hyperbolic manifolds.  Our first main result is:
\begin{intro_theorem}\label{slope}  Let $X$ be a complex hyperbolic manifold of dimension $n$ whose toroidal compactification $\bar X$ has no orbifold points.  Then $K_{\bar X}+(1-\lambda)D$ is ample for $0<\lambda<\frac{n+1}{2\pi}$.  
\end{intro_theorem} 

\begin{intro_cor}In dimension $n\geq 6$, $K_{\bar X}$ is ample---\emph{i.e.} $\bar X$ is the canonical model of $X$.
\end{intro_cor}

Of course, if $X$ is already compact, then $K_X$ is clearly ample.  Theorem \ref{slope} is proven by showing that the volume of a curve near a cusp is bounded by its multiplicity along the corresponding boundary divisor (see Propositions \ref{volumeineq} and \ref{positivity}) with a coefficient depending only on the associated parabolic stabilizer.   The bound allows us to translate group-theoretic properties of the lattice into the positivity of a divisor in the span of $K_{\bar X}$ and the boundary components.    Shimizu's lemma is enough to conclude Theorem \ref{slope} using only the discreteness of the group; specific information about the parabolic stabilizers gives stronger positivity results.  

The toroidal compactification of a complex hyperbolic manifold satisfies the hypotheses of Theorem \ref{slope} under mild assumptions on the uniformizing group (see Definition \ref{smooth}), and every complex hyperbolic orbifold admits a finite \'etale cover which satisfies this property.  Note that $K_{\bar X}+D$ induces the contraction $\bar X\into X^*$ to the Baily--Borel compactification, and therefore always generates one of the boundary rays of the slice of the nef cone cut out by the plane generated by $K_{\bar X}$ and $D$.  It is an interesting question in general for toroidal compactifications (not necessarily of hyperbolic manifolds) to determine the slope of the opposite boundary ray, and Theorem \ref{slope} shows that in this case it grows uniformly with dimension.  

Theorem \ref{slope} implies that hyperbolic manifolds in dimensions $n\geq 6$ are of general type (in fact $K_{\bar X}$ being ample is much stronger), but this need not be true in low dimensions.  Indeed, any rational curve with at least three punctures or any elliptic curve with at least one puncture is hyperbolic, so every Kodaira dimension can arise in dimension 1.  A famous series of examples due to Hirzebruch \cite{hirz} shows that there are also infinitely many smooth hyperbolic surfaces with Kodaira dimension 0 (see Example \ref{hirzy}).  We give an independent algebraic argument that in fact hyperbolic manifolds of dimension $n\geq 3$ are of general type, thereby showing that there is no higher-dimensional analog of Hirzebruch's construction:

\begin{intro_theorem}\label{kodaira}  Let $X$ be a complex hyperbolic manifold of dimension $n\geq 3$ whose toroidal compactification $\bar X$ has no orbifold points.  Then $X$ is of general type.
\end{intro_theorem}
Thus, $K_{\bar X}$ is big; $K_{\bar X}$ is also nef by a recent theorem of Di Cerbo--Di Cerbo \cite{cerbocanon} (see Theorem \ref{minimal} below).  These two facts together imply an interesting consequence to the birational geometry of such varieties:  by the basepoint-free theorem \cite[Theorem 3.3]{kollarmori}, $K_{\bar X}$ is in fact semi-ample.  In general, the abundance conjecture asserts that some multiple of $K_Y$ is basepoint-free for any smooth minimal projective variety $Y$.

\begin{intro_cor}  In dimension $n\geq 3$, $\bar X$ satisfies the abundance conjecture---\emph{i.e.} $K_{\bar X}$ is semi-ample.
\end{intro_cor}
Theorems \ref{slope} and \ref{kodaira} also improve Parker's bound \cite{parker} on the number of cusps of a complex hyperbolic manifold of fixed volume:
\begin{intro_cor}\label{cuspbounds}  If $k$ is the number of cusps of $X$, then
\[\frac{\vol(X)}{k}\geq \frac{2^n}{n}\]
\end{intro_cor}
Theorem \ref{kodaira} gives a slightly better bound in low dimensions, see Corollary \ref{manycusps} (this is also observed in \cite{cerbocanon}).  The bound of Corollary \ref{cuspbounds} is in fact equal to Parker's bound for uniformizing groups whose parabolic subgroups are unipotent \cite[Theorem 3.1]{parker}, though Corollary \ref{cuspbounds} applies to a larger class of lattices (see also the discussion after Corollary \ref{manycusps}).  This is interesting because Parker's method cannot give the same bound in this case.  The main error in Parker's general result comes from bounding the minimal index of a Heisenberg lattice in the stabilizer of a cusp, which does not appear here.

With Theorem \ref{kodaira} in place, we can ask if $\bar X$ satisfies the Green--Griffiths conjecture, which says that if $Y$ is a projective variety of general type, then there is a strict subvariety $Z\subset Y$ such that every nontrivial entire map $\C\into Y$ has image in $Z$.  In this case we say that $Z$ is the exceptional locus  By a theorem of Nadel \cite{nadel}, it is not difficult to show that some finite cover of $\bar X$ satisfies the conjecture; Theorems \ref{slope} and \ref{kodaira} allow us to conclude that ``most" covers do:
\begin{intro_cor}\label{brody}With $X$ as in Theorem \ref{slope}, let $X'\into X$ be a finite \'etale cover that ramifies at each boundary component to order $\ell$.  Then $\bar X'$ satisfies the Green--Griffiths conjecture with the boundary as exceptional locus if: 
\begin{enumerate}
\item $\ell\geq 2$ and $n=3$;
\item $\ell\geq 3$ and $n=4,5$;
\item $\ell\geq 4$ and $n\geq 6$.
\end{enumerate}
\end{intro_cor}

Finally, Theorem \ref{slope} substantially improves a variety of results about complex hyperbolic manifolds that have been proven recently using the algebraic geometry of toroidal compactifications.  These methods use as input the positivity of divisors of the form $K_{\bar X}+(1-\lambda)D$; for $\lambda=0$ it comes for free on any toroidal compactification.  Di Cerbo--Di Cerbo \cite{cerboeff} have systematically studied effectivity results that follow from this positivity in the range $\lambda\in [0,2/3]$ (or more recently for $\lambda\in[0,1]$ in \cite{cerbocanon}), including bounds on the number, degree, and Picard rank of hyperbolic manifolds of a given volume.  For most of these results, simply plugging Theorem \ref{slope} into their argument provides a better result, and we choose to leave these modifications to the reader.

The multiplicity bounds of Propositions \ref{relative} and \ref{volumeineq} can more generally be proven for quotients of a bounded symmetric domain by a rank one lattice, and we investigate the implications to Hilbert modular varieties in \cite{HMV}.

\subsection*{Outline}  In Section \ref{background} we collect some background on hyperbolic manifolds and their toroidal compactifications.  In Section \ref{bound} we prove the volume bound on the multiplicity of curves along the boundary.  These bounds are similar to those proven by Hwang--To \cite{hwangto1} for \emph{interior} points of locally symmetric varieties.   We provide a self-contained algebraic proof of Theorem \ref{kodaira} in Section \ref{algebraic}; in Section \ref{koddim} the result is subsequently strengthened to Theorem \ref{slope} using the multiplicity bound of Section \ref{bound}, and we prove Corollaries \ref{cuspbounds} and \ref{brody}.\subsection*{Acknowledgements}
The first named author would like to thank M. Stover and G. Di Cerbo for many useful conversations, and G. Di Cerbo in particular for introducing the authors to some of the open problems in the field.  This paper was written during the first named author's visit to Columbia University, and he is grateful for their hospitality.
\section{Background}\label{background}
The hyperbolic $n$-ball is the domain
\[\B=\B^n=\{z\in \C^n||z|^2<1\}\]
It has holomorphic automorphism group $\PU(n,1)$ and Bergman metric
\[h=ds_\B^2=4\cdot\frac{(1-|z|^2)\sum_i dz_i\otimes d\bar z_i+(\sum_i \bar z_idz_i)\otimes(\sum_iz_id\bar z_i)}{(1-|z|^2)^2}\]
of constant sectional curvature $-1$.  With this normalization, $\Ric(h)=-\frac{n+1}{2}h$, and the associated K\"ahler form is $\omega_\B=\frac{1}{2}\Imag ds_\B^2$.  
%

Let $\Gamma\subset \PU(n,1)$ be a cofinite-volume discrete subgroup and $X=\B/\Gamma$.  $X$ naturally has the structure of an orbifold; every elliptic element of $\Gamma$ is torsion, so if $\Gamma$ is torsion-free $X$ is a smooth complex manifold.  $\Gamma$ always admits a finite index torsion-free (in fact neat) subgroup, by \cite{amrt} in the arithmetic case and  \cite{hummel} in general.  Henceforth we will typically only consider $\Gamma$ torsion-free, and we will refer to such $X$ as torsion-free ball quotients.  

The cusps of $X$ are in one-to-one correspondence with the equivalence classes of parabolic fixed points of $\Gamma$, and the Baily--Borel compactification $X^*$ is a normal projective variety obtained by adding one point for each cusp (\cite{bb} in the arithmetic case; \cite{sy} in general).  $X$ also admits a unique orbifold toroidal compactification $\bar X$ by \cite{amrt} in the case of an arithmetic lattices $\Gamma$ and by \cite{mok} in general.  If $\bar X$ has no orbifold points (see Definition \ref{smooth}), then it is a smooth projective variety and each connected component $E$ of the boundary divisor $D$ is an \'etale quotient of an abelian variety whose normal bundle $\O_E(E)$ is anti-ample.  If the parabolic subgroups of $\Gamma$ are unipotent (in particular if $\Gamma$ is neat), the boundary $D$ is a disjoint union of abelian varieties.  In any case, the log-canonical divisor $K_{\bar X}+D$ is semi-ample and induces a birational map $\bar X\into X^*$ which is an isomorphism on the open part $X$ and contracts each boundary component $E$ to the point of $X^*$ compactifying the corresponding cusp. 

The hermitian metric $ds_\B^2$ descends to $X$ and extends to a ``good" singular hermitian metric on the log-tangent bundle $T_{\bar X}(-\log D)$ by a theorem of Mumford \cite{mumfordhp}.  Likewise, there is a natural singular hermitian metric on the log-canonical bundle $\omega_{\bar X}(D)$, and integration against the K\"ahler form $\omega_X$ on the open part represents (as a current) a multiple of the first Chern class dictated by our choice of normalization:
\begin{equation}\chone(K_{\bar X}+D)=\frac{1}{2\pi}\frac{n+1}{2}[\omega_X]\in H^{1,1}(\bar X,\R)\label{canonical}\end{equation}
For analyzing the boundary behavior in more detail, the Siegel model is more convenient.  Our presentation is taken from Parker \cite{parker}.  Let 
\[\SS=\SS^n=\C^{n-1}\times \R\times\R_{>0}\]
where $\C^{n-1}$ is endowed with the standard positive definite hermitian form\footnote{Our hermitian forms are linear in the first variable.} $(\cdot,\cdot)$.  We use coordinates $(\zeta,v,u)$, and note that holomorphic coordinates in this model are given by $\zeta$ and 
\[z=v+i(|\zeta|^2+u)\]
whence
\[\SS=\{(\zeta,z)\in\C^{n-1}\times \C\mid \Imag z>|\zeta|^2\}\]
The Siegel model comes with a preferred cusp at infinity whose parabolic stabilizer $G_\infty$ contains the group of Heisenberg isometries $U_\infty:=\U(n-1)\ltimes \frak{N}$ acting only on the first two coordinates $\C^{n-1}\times\R$:  Heisenberg rotations $\U(n-1)$ act on $\C^{n-1}$ in the usual way and Heisenberg translations $\frak{N}\cong \C^{n-1}\times\R$ act via
\[(\tau,t):(\zeta,v)\mapsto (\zeta+\tau,v+t+2\Imag(\tau,\zeta))\]
For completeness, we note that in the holomorphic coordinates this is:
\[(\tau,t):(\zeta,z)\mapsto \left(\zeta+\tau,z+t+i|\tau|^2+2i(\zeta,\tau)\right)\]
We denote by $(A,\tau,t)$ the transformation which first rotates by $A\in U(n-1)$ and then translates by $(\tau,t)$.  $\frak{N}$ is a central extension of the group $\C^{n-1}$ of translations on the first coordinate by the group $\R$ of translations in the second coordinate.  We call translations of the form $(0,t)$ \emph{vertical} translations, and note that the subgroup $T_\infty\subset G_\infty$ of vertical translations is the center.  The group $U_\infty/T_\infty$ is identified with the group of affine unitary transformations of $\C^{n-1}$ via projection to the $\zeta$ coordinate.

The subgroup $U_\infty\subset G_\infty$ can be thought of as the stabilizer of the \emph{height} coordinate $u$, and $-2\log u$ is a potential for the K\"ahler form:

\begin{lemma} $\omega_\SS=-2i\partial\bar\partial \log u$.
\end{lemma}
\begin{proof}This follows from a computation and the fact that in the Siegel model the hermitian metric is
\[ds_\SS^2=\frac{du^2+(dv-2\Imag (\zeta,d\zeta))^2+4u(d\zeta,d\zeta)}{u^2}\]
(see \emph{e.g.} \cite{parker}).
\end{proof}

The horoball $B(u)$ of height $u$ centered at the cusp at infinity is defined to be the set $$B(u)=\C^{n-1}\times\R\times (u,\infty)$$It is clearly preserved by $U_\infty$.  The remaining generator of $G_\infty$ is a one-dimensional torus which scales $(\zeta,v,u)\mapsto(a\zeta,a^2v,a^2u)$, and this scales the horoball of height $u$ in the obvious way.

Now suppose $\Gamma$ has a parabolic fixed point at infinity, and let $\Gamma_\infty=\Gamma\cap G_\infty$ be its stabilizer.  For any horoball $B(u)$ centered at infinity, define the horoball neighborhood $V(u):=B(u)/\Gamma_\infty$.  Note that at some sufficiently large height $u$, $V(u)$ injects into $X$  by Shimizu's lemma.  We call the smallest such $u$ the height $u_\infty$ of the cusp. 
The partial quotient by the vertical translations $\Theta_\infty=\Gamma\cap T_\infty$ is given by the map
\[\SS\into \C^{n-1}\times \Delta^*:(\zeta,z)\mapsto (\zeta,e^{2\pi iz/t_\infty})\]
The cusp is compactified by taking the interior closure in $\C^{n-1}\times\Delta$, which corresponds to adding a boundary component of the form $D_\infty=\C^{n-1}\times 0/\Lambda_\infty$, where $\Lambda_\infty:=\Gamma_\infty/\Theta_\infty$ is identified with a discrete group of affine unitary transformations.  $\Lambda_\infty$ naturally comes with a character 
\[\chi_\infty:(A,\tau,t)\mapsto e^{2\pi i t/t_\infty}\]
which encodes the action on the second factor of $\C^{n-1}\times \Delta$.  Note that $\chi_\infty$ has finite order since $\Gamma$ is discrete, and we call its order $m_\infty$.  Of course $\chi_\infty$ is trivial on Heisenberg translations, so $m_\infty$ divides the minimum index of a Heisenberg lattice in $\Gamma_\infty$.
\begin{definition}\label{smooth}  We say that $\Gamma$ is \emph{torsion-free at infinity} if $\Gamma$ is torsion-free and $\Lambda_\infty$ is torsion-free for each parabolic fixed point $q_\infty$.  Equivalently, $\Gamma$ is torsion-free at infinity if the orbifold toroidal compactification $\bar X$ of $X=\B/\Gamma$ has no orbifold points.
\end{definition}
Note that the condition that $\bar X$ have smooth coarse space is slightly weaker, as we only need every parabolic element to have one non-identity eigenvalue.  Neat groups are clearly torsion-free at infinity, as are groups all of whose parabolic subgroups are unipotent.  Every $\Gamma$ contains a finite-index neat subgroup, so clearly every complex hyperbolic orbifold $X$ has a finite \'etale cover $X'$ whose uniformizing group is torsion-free at infinity.

If $\Gamma$ is torsion-free at infinity, then the residual quotient of $\C^{n-1}\times \Delta$ by $\Lambda_\infty$ is \'etale, so locally around the boundary we have coordinates $\zeta,q=e^{2\pi i z/t_\infty}$ and the boundary is cut out by $q=0$.  $V(u)$ is then identified with a neighborhood of the zero section in the normal bundle $\O_{D_\infty}(D_\infty)$ (\emph{cf.} \cite{mok}).  In general, at a fixed point $\zeta\in \C^{n-1}\times 0$, if $m_\zeta$ is the order of $\chi_\infty$ restricted to the stabilizer of $\zeta$, then $q^{m_\zeta}$ locally descends to a function on the coarse space of $\bar X$ which vanishes along the boundary.

\section{Boundary multiplicity inequalities}\label{bound}

Let $X=\B/\Gamma$ be a torsion-free ball quotient and suppose $q_\infty$ is a parabolic fixed point of $\Gamma$ with stabilizer $\Gamma_\infty=\Gamma\cap G_\infty$.  By considering the Siegel model associated to $q_\infty$, we have by the previous section horoball neighborhoods $V(u)\subset X$ for all $u<u_\infty$, where $u_\infty$ is the height of $q_\infty$.  Let $\bar V(u)$ be the interior closure of $V(u)$ in the toroidal compactification $\bar X$. 

We first show that the volume of an analytic subvariety of the horoball neighborhood $\bar V(u)$ scales as the height of the horoball drops. 

\begin{prop}\label{relative}Let $Y$ be an irreducible $k$-dimensional analytic subvariety of $\bar V(u)$ not contained in the boundary.  Then $$u^k\vol(Y\cap V(u))$$ is a non-increasing function of $u>u_\infty$.

\end{prop}

Of course, Proposition \ref{relative} is equally true in the orbifold setting, since we may simply pass to a torsion-free cover.

Before the proof we recall a lemma of Demailly \cite{demaillybook}.  Let $X$ be a complex manifold and $\phi:X\into[-\infty,\infty)$ a continuous plurisubharmonic function.  Define 
\[B_\phi(r)=\{x\in X\mid \phi(x)<r\}\]
We say $\phi$ is semi-exhaustive if the balls $B_\phi(r)$ have compact closure in $X$.  Further, for $T$ a closed positive current of type $(p,p)$, we say $\phi$ is semi-exhaustive on $\Supp T$ if the same is true for $B_\phi(r)\cap \Supp T$.  In this case, the integral
\[\int_{B_\phi(r)}T\wedge (i\partial\bar \partial\phi)^p:=\int_{B_\phi(r)}T\wedge (i\partial\bar \partial\max(\phi,s))^p\] 
is well-defined and independent of $s<r$ \cite[\S III.5]{demaillybook} (see also \cite{hwangto1}).  We then have the
\begin{lemma}[Formula III.5.5 of \cite{demaillybook}]\label{demailly}For any convex increasing function $f:\R\into\R$,
\[\int_{B_\phi(r)}T\wedge(i\partial\bar \partial f\circ\phi)^p=f'(r-0)^p\int_{B_\phi(r)}T\wedge(i\partial\bar \partial\phi)^p\]
where $f'(r-0)$ is the derivative of $f$ from the left at $r$.
\end{lemma}

\begin{proof}[Proof of Proposition \ref{relative}]
\begin{align*}
\vol(Y\cap V(u_0))&=\frac{1}{k!}\int_{Y\cap  V(u_0)}\omega_X^k\\
&=\frac{1}{k!}\int_{Y\cap V(u_0)}(i\partial\bar\partial (-2\log u))^k\\
&=\frac{1}{k!}\int_{V(u_0)}(i\partial\bar\partial (-2\log u))^k\wedge [Y]\\
&=\frac{2^ku_0^{-k}}{k!}\int_{V(u_0)}(i\partial\bar\partial (-u))^k\wedge [Y]\\
&=\frac{2^ku_0^{-k}}{k!}\int_{Y\cap V(u_0)}(i\partial\bar\partial (-u))^k\\
\end{align*}
As $-u$ is plurisubharmonic,
\[u_0\vol(Y\cap V(u_0))=\frac{2^k}{k!}\int_{Y\cap V(u_0)}(i\partial\bar\partial (-u))^k\]
is a non-increasing function of $u_0$ (the horoballs $V(u_0)$ \emph{shrink} as $u_0$ grows).
\end{proof}
Taking the limit of Proposition \ref{relative} as $u\into 0$ yields a bound on the multiplicity of a curve at the boundary in terms of its volume in a horoball neighborhood.
\begin{prop}\label{volumeineq}  Assume $\Lambda_\infty=\Gamma_\infty/\Theta_\infty$ is torsion-free and let $t_\infty$ be the length of the smallest vertical translation $(0,t_\infty)\in \Gamma_\infty$.  For any irreducible 1-dimensional analytic subvariety $C$ of $\bar V(u)$ not contained in the boundary and any $u>u_\infty$, we have
\begin{equation}\vol(C\cap V(u))\geq \frac{t_\infty}{u}\cdot (C.D_\infty)\notag\end{equation}
where $D_\infty$ is the divisor compactifying $q_\infty$ in the toroidal compactification.
\end{prop} 
\begin{proof}

From the proof of the previous proposition, we just need to compute
\[2\cdot \lim_{u_0\into \infty}\int_{C\cap V(u_0)}i\partial\bar\partial (-u)\]

For $u_0$ sufficiently large, $C\cap V(u_0)$ is a union of pure 1-dimensional analytic sets, each component of which is normalized by a disk $f_j:\Delta_j\into C\cap V(u_0)$.  We may assume $f_j(0)=x_j\in D_\infty$ and that $f_j|_{\Delta_j^*}$ is an isomorphism onto an open set of $C$.  $q=e^{2\pi i z/t_\infty}$ is a local defining equation for $D_\infty$ and we have
\[C.D_\infty=\sum_j \ord f_j^* q\]
Now for sufficiently large $u_0$, we have 
\[\int_{C\cap V(u_0)}i\partial\bar\partial (-u)=\sum_j\int_{\Delta_j}f_{j}^*\partial\bar\partial (-u)\]
but of course
\begin{align*}\int_{\Delta_j}f_{j}^*\partial\bar\partial (-u)&\geq \pi\cdot\nu(f_j^*(-u),0)
\end{align*}
If $t$ is a uniformizer for $\Delta_j$ at $0$, then we compute 
\begin{align*}\nu(f_j^*(-u),0)&=\liminf_{t\into 0} \frac{f_j^*(-u)}{\log |t|}\\
&=\liminf_{t\into 0}\frac{1}{\log|t|}\cdot f_j^*\left(|\zeta|^2+\frac{t_\infty}{2\pi}\cdot\log|q|\right)\\
&=\frac{t_\infty}{2\pi}\cdot \ord f_j^* q\end{align*}
\end{proof}
\begin{remark}\label{modification}  If we don't assume $\Lambda_\infty$ is torsion-free, then we've proven 
\begin{equation}\vol(C\cap V(u))\geq \frac{t_\infty}{u}\cdot \widetilde{(C.D_\infty)}\notag\end{equation}
where we've defined a weighted intersection product
\[\widetilde{(C.D_\infty)}=\sum_{x\in D_\infty}\frac{1}{m_x}(C.D_\infty)_x\]
Here $(C.D_\infty)_x$ is the contribution of $x$ to the usual intersection product on the coarse space of $\bar X$, and $m_x$ is the quantity defined at the end of Section \ref{background}.  In particular, we at least have 
\begin{equation}\vol(C\cap V(u))\geq \frac{t_\infty}{m_\infty u}\cdot (C.D_\infty)\notag\end{equation}
\end{remark}

\begin{remark} Proposition \ref{volumeineq} is sharp in the sense that a union of vertical complex geodesics will realize the equality.  A vertical complex geodesic is a copy of the upper half-plane $\H\subset \SS$ embedded as $\zeta=0$ (or a horizontal translate thereof), and the intersection of $\H$ with the horoball $B(u)$ is $\H_{>u}=\{z\in\H\mid \Imag z>u\}$.  The resulting curve $C$ in $V(u)$ is the quotient of $\H_{>u}$ by real translation by $t_\infty$ and therefore has $\vol(C\cap V(u))=t_\infty/u$.  Finally, we have $(C.D_\infty)=1$, as $C$ is uniformized by $0\times\Delta$ in the partial quotient $\C^{n-1}\times \Delta$.
\end{remark}

Proposition \ref{volumeineq} is analogous to the multiplicity bound proven by Hwang--To \cite{hwangto1} for an \emph{interior} point $x$ of a quotient of a bounded symmetric domain.  They show for a $k$-dimensional subvariety that
\[\vol(Y\cap B(x,r))\geq \vol(D(r))^{k}\cdot\mult_xY\]
where $B(x,r)$ is an isometrically embedded hyperbolic ball around $x$ of radius $r$ and $D(r)$ is the volume in $B(x,r)$ of a complex geodesic through $x$.  One can show in this case a relative version as in Proposition \ref{relative} as well.

We could have proven Proposition \ref{volumeineq} directly by methods more analogous to \cite{hwangto1}.  As in Section \ref{background}, the hermitian metric $h$ on $\omega_X$ extends to a singular hermitian metric $\bar h$ on $\omega_{\bar X}(D)$.  We form a different singular metric by twisting by a function $e^{-\phi}$ supported on $V(u_0)$ so that $e^{-\phi}\bar h$ has positive curvature form and Lelong number $\frac{t_\infty}{u_0}$ at every point of the boundary.  As $\bar h$ is given by $e^{2\log u}$ on $V(u)$, taking $\phi$ so that $\phi-2\log u$ approximates the tangent line to $-2\log u$ at $u_0$ will achieve this.  We choose instead to derive Proposition \ref{volumeineq} from Proposition \ref{relative} because the latter statement is interesting (and useful, \emph{cf.} \cite{HMV}) in its own right.

\section{Kodaira dimension}\label{algebraic}
In this section we prove Theorem \ref{kodaira}.  Our proof will actually be independent of the results of Section \ref{bound} and entirely algebraic.  In Section \ref{koddim} we'll use the multiplicity bound from Proposition \ref{volumeineq} to prove a much stronger statement about the positivity of divisors of the form $K_{\bar X}+(1-\lambda)D$, and in particular show that $K_{\bar X}$ is ample for $n\geq 6$.  Such results seem to be more difficult to prove algebraically.

We begin with an example due to Hirzebruch \cite{hirz} for context.
\begin{example}\label{hirzy}  Let $\zeta=e^{2\pi i/3}$ and $E=\C/\Z[\zeta]$ be the elliptic curve with $j=0$.  Consider the blow-up $S$ of $E\times E$ at the origin $0\in E\times E$.  We have $K_S\equiv F$ where $F$ is the exceptional divisor.  If we let $D$ be the union of the strict transforms of the fibers $E\times0$, $0\times E$, and the graphs of $1,-\zeta\in \Z[\zeta]$, then the complement $U=S\smallsetminus D$ is uniformized by $\B^2$ by a theorem of Yau, since we compute
\[3=(F+D)^2=\chone{(\omega_S(\log D))}^2=3\chtwo{(\Omega^1_S(\log D))}=3\chi(U)\]
and $K_S+D$ is big and nef.  It follows that $S$ is the toroidal compactification of $U$ with boundary $D$.
\end{example}
Hirzebruch's example shows that the toroidal compactification of a torsion-free (in fact neat) ball quotient in dimension 2 may be non-minimal (\emph{i.e.} $K_{\bar X}$ is not nef) and may have Kodaira dimension 0.  Blow-ups of $E\times E$ at special configurations of points for other elliptic curves $E$ yield infinitely many such examples.

The main goal of this section is to show that neither of these phenomena can occur in higher dimensions, and in particular that every complex hyperbolic manifold of dimension $\geq 3$ is of general type.  Recall that a quasiprojective variety $X$ is of general type if some projective compactification $X'$ has maximal Kodaira dimension, $\kappa(X')=n$.  

\begin{prop}\label{general}Let $\Gamma$ be torsion-free at infinity and $X=\B/\Gamma$.  Then $\bar X$ is of general type.  In fact, $K_{\bar X}$ is big and nef.
\end{prop}
The nefness is a recent result of Di Cerbo--Di Cerbo \cite{cerbocanon}.  Recall that the abundance conjecture asserts that for a smooth projective variety $Y$ with $K_Y$ nef, then in fact $K_Y$ is semi-ample.  By the basepoint-free theorem \cite[Theorem 3.3]{kollarmori}, we can conclude that this is the case for toroidal compactifications of complex hyperbolic manifolds:

\begin{cor}With $X$ as above, $\bar X$ satisfies the abundance conjecture, \emph{i.e.} $K_{\bar X}$ is semi-ample.
\end{cor}

As another application, we can show every complex hyperbolic threefold with a smooth toroidal compactification has a cover which satisfies the Green--Griffiths conjecture.  Recall that this conjecture asserts that if $Y$ is a projective variety of general type, then there is a strict subvariety $Z\subset Y$ such that every nontrivial entire map $\C\into Y$ has image contained in $Z$.  In this case we say that $Z$ is the exceptional locus.

\begin{cor}\label{brody1}  With $X$ as above, let $\pi:X'\into X$ be a finite \'etale cover.  Then $\bar X$ satisfies the Green-Griffiths conjecture with exceptional locus $D$ if:
\begin{enumerate}
\item $n=3$ and $\pi$ ramifies along every boundary component;
\item $n=4,5$ and $\pi$ ramifies to order at least 3 along every boundary component.
\end{enumerate}
\end{cor}

The remaining case of Corollary \ref{brody} follows from Corollary \ref{brody2} of the next section.

\begin{proof}[Proof of Corollary \ref{brody1} assuming Proposition \ref{general}]  Suppose $Y$ is a finite-volume quotient of a bounded symmetric domain whose holomorphic sectional curvature is $\leq-\gamma$ (with the normalization $\Ric(h)=-h$) for some $\gamma\in\Q$.  Then by a theorem of Nadel \cite[Theorem 2.1]{nadel}, if $\bar Y$ is a smooth toroidal compactification and $K_{\bar Y}+(1-1/\gamma)D$ is big, then every entire map $\C\into\bar Y$ has image contained in the boundary.

For us, $\gamma=\frac{2}{n+1}$.  If $\pi$ ramifies to order $\ell$ along each boundary component then $\pi^*D\geq \ell D'$, and
\[\pi^*K_{\bar X}=K_{\bar X'}+D'-\pi^*D\leq K_{\bar X'}+(1-\ell)D'\]
$K_{\bar X}$ is big by the proposition, so the right hand side is as well, and this is enough for parts (1) and (2).
\end{proof}

Our proof of Propostion \ref{general} will only require the coarse space of $\bar X$ to be smooth up until the last step in Lemma \ref{kodairabigenough}.  For completeness, we first summarize the argument of \cite{cerbocanon} for the nefness of $K_{\bar X}$ using the cone theorem and bend-and-break.

Given a smooth curve $C$, a projective variety $Y$, a set of points $S\subset C$, and a map $f|_S:S\into Y$, we denote by $\Hom(C,Y;f|_S)$ the scheme parametrizing maps $f:C\into Y$ restricting to $f|_S$ along $S$.  Recall that bend-and-break says (\emph{cf.} \cite[Propositions 3.1 and 3.2]{debarre})

\begin{prop}[Bend-and-break]\label{bendy}Let $Y$ be a projective variety.
\begin{enumerate}
\item For any map $f:C\into X$ from a smooth pointed curve $0\in C$ and any (quasiprojective) curve $B\subset \Hom(C,X;f|_0)$ containing $f$ along which $f_b(C)$ is not constant, $f_b(C)$ has a limit with a rational component;

\item For any map $f:\P^1\into X$ and any (quasiprojective) curve $B\subset \Hom(\P^1,X;f|_{\{0,\infty\}})$ containing $f$ along which $f_b(\P^1)$ is not constant, $f_b(\P^1)$ has a limit which is a reducible or multiple rational curve.
\end{enumerate}
\end{prop}

The key idea for us is that an extremal $K_{\bar X}$-negative rational curve $f:\P^1\into \bar X$ must intersect the boundary $D$ in at least $3$ points since $X$ is uniformized by a bounded domain.  On the other hand, $f:\P^1\into \bar X$ deforms, since for any component $B$ of $\Hom(\P^1,\bar X)$ containing $f$,
\begin{equation}\dim B\geq -K_{\bar X}.f(\P^1)+\dim \bar X\label{dimension}\end{equation}
As long as $n\geq 3$, then $\dim B\geq 4$, and in the Baily--Borel compactification $X^*$ we have a family of rational curves with 3 fixed points, so by bend-and-break $f(\P^1)$ is algebraically equivalent to a reducible or multiple rational curve.  By induction we have a contradiction.  Thus, $\bar X$ can only be non-minimal if $n=2$:
\begin{theorem}[Theorem 1.1 of \cite{cerbocanon}]\label{minimal}$K_{\bar X}$ is nef if $n\geq 3$.

\end{theorem}

\begin{remark}  Given Lemma \ref{minimal}, to prove Proposition \ref{general} it would be enough to show $K_{\bar X}^n>0$.  $L=K_{\bar X}+D$ induces the contraction to the Baily--Borel compactification, and 
\[K_{\bar X}^n=L^n+(-D)^n\]
For any component $E$ of the boundary, $-(-E)^n$ computes the rate of growth of the volume of a horoball neighborhood of $E$, and $L^n$ computes the global volume of $X$, up to a normalization.  The best known bounds on the size of distinct horoball neighborhoods give bigness for $n\geq 6$ (but only in the case of neat quotients); one could conceivably finish the proof of Proposition \ref{general} by a case by case analysis as in Parker \cite{parker}.  We instead pursue the algebraic line of attack.
\end{remark}

We now need to understand curves $C$ for which $K_{\bar X}.C=0$; we call such curves $K_{\bar X}$-trivial.  We call an (irreducible) curve $C$ rigid if no component of $\Hom(\tilde C,\bar X)$ containing the normalization $\tilde C\into \bar X$ has dimension greater than the dimension of the infinitesimal automorphism group $\dim H^0(\tilde C, T_{\tilde C})$ (that is, $3,1,0$ for $g(\tilde C)=0,1,\geq 2$ respectively).
\begin{lemma}\label{trivial}  For $n=3$, any $K_{\bar X}$-trivial rational curve is algebraically equivalent to a sum of rational curves $\sum_i C_i$ with each $C_i$ rigid.  For $n\geq 4$, there are no $K_{\bar X}$-trivial rational curves.  
\end{lemma}
\begin{proof}  If $f:\P^1\into \bar X$ has $K_{\bar X}.f(\P^1)=0$, then for any component $B$ of $\Hom(\P^1,\bar X)$ containing $f$ for which $\dim B\geq 4$, $f(\P^1)$ deforms to a reducible or multiple rational curve $\sum_i C_i$ by bend-and-break.  Since $K_{\bar X}$ is nef, we have $K_{\bar X}.C_i=0$ for each $i$.  By induction we can assume the $C_i$ are rigid.  By \eqref{dimension}, if $n\geq 4$ there are no rigid $K_{\bar X}$-trivial rational curves.
\end{proof}
The following corollary is not needed, but interesting nonetheless.
\begin{cor}  For $n\geq 4$, any $K_{\bar X}$-trivial curve is rigid.
\end{cor}
\begin{proof}  If a $K_{\bar X}$-trivial curve $C$ is not rigid, then there is a positive dimensional component of $\Hom(\tilde C,\bar X)$ containing $\tilde C\into \bar X$.  Note that $C$ is not contained in the boundary since $K_{\bar X}|_D$ is ample, and further $D.C>0$.  Therefore $\tilde C\into X^*$ deforms with a fixed point, so by bend-and-break it deforms to a curve with a rational component which must be $K_{\bar X}$-trivial, contradicting the lemma.
\end{proof}


If we assume the abundance conjecture, then Lemma \ref{trivial} is enough to conclude Proposition \ref{general}.  Indeed, by Lemma \ref{minimal}, $K_{\bar X}$ would then be semi-ample, so let $f:\bar X\into Z$ be the fiber space induced by $|mK_{\bar X}|$ for $m\gg 0$.  For any fiber $F$ and any curve $C\subset F$, $K_{\bar X}.C=0$ whereas $D|_D\equiv -K_{\bar X}|_D$ is anti-ample.  We must therefore have $\dim(D\cap F)=0$, so $\dim F=n-\kappa(\bar X)\leq 1$.  But if $\kappa(\bar X)=n-1$, then the general fiber $F$ is a $K_{\bar X}$-trivial elliptic curve.  Again because $X$ is uniformized by a bounded domain, $E.F\geq 1$ for some component $E$ of the boundary.  Taking a curve in $C\subset E$ such that the fibers of $\mathcal{E}=f^{-1}(f(C))$ have fixed $j$-invariant, $C$ is a multisection of $\mathcal{E}/f(C)$, so base-changing to $C$ we have an isotrivial family $\mathcal{E}/C$ with a section.  Projecting to $X^*$, this is a family of maps $F\into X^*$ fixing $0\in F$, and by bend-and-break there is a $K_{\bar X}$-trivial rational curve.  For $n\geq 4$, this contradicts Lemma \ref{trivial}.  

For $n=3$ the abundance conjecture is known \cite{abundancethree}.  In this case there are finitely many rigid $K_{\bar X}$-trivial rational curves $R$ with $D.R\leq D.F$, and these curves are contracted via $f$ to finitely many points of $Z$.  If we choose $C$ to miss these points, then we again have a contradiction.

In higher dimensions, the best that is currently known is a rational version of the abundance conjecture.  Recall that for $M$ a nef line bundle on a normal projective variety $ Y$, there is a nef reduction map $f:Y\dashto Z$ to a normal variety $Z$ \cite{reduce}.  This map is the unique (up to birational equivalence on $Z$) dominant rational map with connected fibers such that
\begin{enumerate}
\item $f$ is ``almost holomorphic" in the sense that if $U\subset Y$ is the maximal open set on which $f$ is defined, $f:U\into Z$ has a proper fiber (and therefore the general fiber is proper);
\item $M$ is numerically trivial on all proper fibers of dimension $\dim Y-\dim Z$;
\item For a general point $y\in  Y$ and any irreducible curve $C$ through $x$ with $\dim f(C)=1$ we have $L.C>0$.
\end{enumerate}
We then call $n(M):=\dim Z$ the \emph{nef dimension} of $M$ and $n(Y):=n(K_Y)$ the \emph{nef dimension} of $Y$ if $K_Y$ is nef. 

\begin{lemma}\label{nefdim}$\bar X$ has maximal nef dimension if $n\geq 3$.
\end{lemma}
\begin{proof}
For $n=3$, we are done by the argument above.  For $n\geq 4$, take $M=K_{\bar X}$ on $Y=\bar X$ and let $F$ be a general fiber of the nef reduction.  $K_F=K_{\bar X}|_F$ is numerically trivial, so $D|_F$ is ample, and therefore it must again be the case that $\dim (F\cap D)=0$.   This can only happen if $\dim F\leq1$.  If $\dim F=0$, we're done, while if $\dim F=1$, $F$ is an elliptic curve, and the argument above provides the contradiction.
%
%

\end{proof}

We would like to show that Lemma \ref{nefdim} implies that $K_{\bar X}$ is big, but in general for a nef line bundle $M$ it is only the case that
\[n(M)\geq \nu(M)\geq \kappa(M)\]
where $\nu(M)$ is the numerical dimension and $\kappa(M)$ is the Iitaka dimension of $M$.  If $M=K_Y$ is the canonical bundle of a smooth projective variety $Y$, then the abundance conjecture implies all three are equal, but we can already see that maximal nef dimension implies bigness assuming $\kappa(X)$ is sufficiently large:

\begin{lemma}\label{nefthenbig}Let $Y$ be a smooth $n$-dimensional projective variety with $K_Y$ nef.  If $n(Y)=n$ and $\kappa(Y)\geq n-2$, then in fact $\kappa(Y)=n$.
\end{lemma}
\begin{proof}
Let $f:Y'\into Z$ be the Iitaka fibration of $K_{Y}$, which admits a birational map $g:Y'\into Y$, and let $F$ be a very general fiber of $f$.  We know that $\dim F=n-\kappa(Y)\leq 2$, that $g^*K_{Y}|_F$ has Iitaka dimension 0, and that $\kappa(F)=0$ (\emph{e.g.} \cite[\S 2.1.C]{Lazarsfeld}).  We also know $g^* K_{ Y}|_F$ is nef and nonzero on every curve through a very general point of $F$ by the assumptions, which immediately implies $\dim F\neq 1$.  If $\dim F=2$, then for some effective divisor $E$, $g^*K_{ Y}|_F+E= K_F$, but $K_F$ is numerically equivalent to a sum of $-1$ curves since $\kappa(F)=0$, by the Enriques-Kodaira classification of surfaces.  Thus there is a curve $C$ in $F$ with $K_F.C=0$ and $C.E\geq 0$ while $g^*K_{Y}.C>0$, which is a contradiction.
\end{proof}

Given Lemmas \ref{nefdim} and \ref{nefthenbig}, the proof of Proposition \ref{general} will be completed by the following
\begin{lemma}\label{kodairabigenough}$\bar X$ has Kodaira dimension $\kappa(\bar X)\geq n-2$ for $n\geq 3$.
\end{lemma}
\begin{proof} Let $L=K_{\bar X}+D$.  $K_{\bar X}$ is nef by Lemma \ref{minimal}, so
\[K_{\bar X}^n=(L-D)^n=L^n+(-D)^n\geq 0\]
If we have strict inequality, then $K_{\bar X}$ is big as it is already nef by Lemma \ref{minimal}.  Thus, we need only treat the case $K_{\bar X}^n=0$, \emph{i.e.} $L^n+(-D)^n=0$.

For $a> b>0$, consider the sequence
\begin{equation}\label{sequence2}0\into \O_{\bar X}(aK_{\bar X}+(b-1)D)\into \O_{\bar X}(aK_{\bar X}+bD)\into \omega^a_D(-(a-b)D)\into 0\end{equation}
where we've used $\O_D(L)\cong \omega_D$.  By Kawamata--Viehweg vanishing neither the middle nor the right term has higher cohomology.  Indeed, for the right term $\omega_D(-D)$ is ample since $\omega_D$ is numerically trivial, while for the middle term 
\[aK_{\bar X}+bD\equiv K_{\bar X}+bL+(a-b-1)K_{\bar X}\]
and since $L$ is big and nef, the same is true of $(a-b-1)K_{\bar X}+bL$ by Lemma \ref{nef}.  

We can conclude two things from this.  First, setting $a=m+1$ and $b=1$, $(m+1)K_{\bar X}$ has no cohomology in degree 2 and higher for $m\geq 1$, and 
\begin{align*}
\dim H^0(\bar X,\O((m+1)K_{\bar X}))&\geq \chi((m+1)K_{\bar X})\\
&= \chi((m+1)K_{\bar X}+D)-\chi( \omega^{m+1}_D(-mD))
\end{align*}
Second, using \eqref{sequence2} for $a=m+1$ and $b=2,\ldots,m$, all three terms have no higher cohomology, and we have
\begin{align*}
\chi((m+1)K_{\bar X}) &=\chi(K_{\bar X}+mL)-\sum_{\ell=0}^m\chi(\omega^{m+1}_D(-\ell D))\\
&=(-1)^n\chi(-mL)-\sum_{\ell=0}^m\chi(\omega^{m+1}_D(-\ell D))
\end{align*}
By Hirzebruch proportionality \cite{mumfordhp}, as $\P^n$ is the compact dual of $\B^n$,
\begin{align*}
\chi(\bar X,-tL)&=\frac{(-L)^n}{(n+1)^n}\cdot\chi(\P^n,\omega_{\P^n}^{-t})+O(1)\\
&= \frac{(-L)^n}{(n+1)^n}\cdot\binom{(n+1)t+n}{n}+O(1)
\end{align*}
Since each component of the boundary is an \'etale quotient of an abelian variety, all of the chern classes of $\Omega_D^1$ vanish numerically, and
\[\chi(D,\omega^{m+1}_D(-\ell D))=-\frac{(-\ell D)^{n-1}}{(n-1)!}\]
Thus,
\begin{align*}
\dim H^0(\bar X,\O((m+1)K_{\bar X}))&\geq\frac{L^n}{(n+1)^n}\binom{(n+1)m+n}{n}+\frac{(-D)^n}{(n-1)!}\sum_{\ell=0}^{m}\ell^{n-1} +O(1)
\end{align*}
The coefficients of $m^n$ and $m^{n-1}$ vanish if $K_{\bar X}^n=0$, and one can compute that the coefficient of $m^{n-2}$ in the first term is
\begin{align*}
\frac{L^n}{(n+1)^n}\frac{(n+1)^{n-2}m^{n-2}}{n!}\cdot \frac{1}{2}\left(\left(\sum_{i=0}^ni\right)^2-\sum_{i=0}^n i^2\right)
&=\frac{L^n}{(n-1)!}\cdot\frac{(n-1)(3n+2)}{24(n+1)}\cdot m^{n-2}
\end{align*}
while that of the second term is
\[\frac{(-D)^n}{(n-1)!}\cdot \frac{n-1}{12}\cdot m^{n-2}\]
using Faulhaber's formula for $\sum^m_{\ell=0}\ell^{n-1}$ (\emph{cf.} \cite{conway}).  As
\[\frac{3n+2}{24}>\frac{n+1}{12}\]
for $n\geq 3$, we are done.
\end{proof}

\section{Ampleness and applications}\label{koddim}

In this section we prove Theorem \ref{slope} and derive Corollaries \ref{cuspbounds} and \ref{brody}.  The proof uses only the multiplicity bound of Proposition \ref{volumeineq}, and is independent of the algebraic arguments of Section \ref{algebraic}.

\begin{prop}\label{nef} Let $\Gamma$ be torsion-free at infinity and $\bar X$ the toroidal compactification of $X=\B/\Gamma$ with boundary $D$.  Then $K_{\bar X}+(1-\lambda)D$ is ample for $0<\lambda<\frac{n+1}{2\pi}$.
\end{prop}

Since $\frac{n+1}{2\pi}>1$ for $n\geq 6$, we deduce:
\begin{cor}  $K_{\bar X}$ is ample if $n\geq 6$.
\end{cor}

It was recently shown in \cite{cerbocanon} that this is always true in dimension $n\geq 3$ up to a cover:  if $X'\into X$ is an \'etale cover ramifying along each boundary component, then $K_{\bar X'}$ is ample.

For the proof of Proposition \ref{nef}, let $q_i$ be the cusps of $X$, and denote by $D_i$ the boundary component of $\bar X$ compactifying $q_i$.  Let $t_i$ be the length of the smallest vertical translation in the stabilizer of $q_i$.

Now suppose for each cusp $q_i$ we choose a horoball height $u_i$ such that:  
\begin{enumerate}
\item[($*$)] each $V(u_i)$ injects into $X$ (\emph{i.e.} $u_i$ is less than the height of $q_i$); 
\item[($**$)] the $V(u_i)$ are all disjoint.
\end{enumerate}  

\begin{prop}\label{positivity} With the above notation, let $L=K_{\bar X}+D$.  Then
\begin{equation}L- \frac{n+1}{4\pi}\sum_i\frac{t_i}{u_i}D_i\label{sumdiv}\end{equation}
is nef.
\end{prop}
\begin{proof} By \eqref{canonical} and Proposition \ref{volumeineq}, the divisor is nef modulo the boundary, but for any component $E$ of the boundary $K_{\bar X}|_E\equiv-E|_E$ is ample.
\end{proof}

\begin{proof}[Proof of Proposition \ref{nef}]
By \cite[Proposition 2.4]{parker}, we can uniformly take $u_i=t_i/2$, and the resulting ``canonical" horoballs satisfy properties $(*)$ and $(**)$.  It follows that $K_{\bar X}+(1-\frac{n+1}{2\pi})D$ is nef.  On the other hand, we already know $K_{\bar X}+(1-\epsilon)D$ is ample for small $\epsilon>0$ (\emph{cf.} \cite{cerboeff}).  The interior of any line drawn between a point of the nef cone and a point in the ample cone is contained in the ample cone, and $K_{\bar X}+tD$ for $t\in[1-\frac{n+1}{2\pi},1-\epsilon]$ is such a line.  
\end{proof}

One immediate application of Proposition \ref{nef} is a bound on the number of cusps of $X$:

\begin{prop} \label{cuspcount} For $\Gamma$ torsion-free at infinity, let $k$ be the number of cusps of $X=\B/\Gamma$.  Then
\[k\leq\frac{(2\pi)^n}{(n+1)^n}\cdot \frac{L^n}{(n-1)!}\]
Further, in dimensions $n=3,4,5$, we have \[k\leq \frac{L^n}{(n-1)!}\]
\end{prop}
\begin{proof}Note that each component of the boundary is an \'etale quotient of an abelian variety so all of the chern classes of $\Omega_D$ vanish numerically.  $D|_D$ is anti-ample, so on the one hand
\[k\leq \frac{D.(-D)^{n-1}}{(n-1)!}=-\frac{(-D)^n}{(n-1)!}=\chi(D,\O_D(-D))\]
but on the other hand if $aL-bD$ is a nef $\R$-divisor for $a,b>0$,
\[0\leq(aL-bD)^n=a^nL^n+b^n(-D)^n\]
Thus
\[k\leq \left(\frac{a}{b}\right)^n\cdot \frac{L^n}{(n-1)!}\]
By Proposition \ref{nef}, we can take $a=1$ and $b=\frac{n+1}{2\pi}$.  By Lemma \ref{minimal} of the next section, for $n=3,4,5$ we can do better with $a=1$ and $b=1$.
\end{proof}
A similar argument is used by Di Cerbo--Di Cerbo to give an improvement to Parker's cusp bound in dimensions 2  \cite{cerbosharp} and 3 \cite{cerboeff}.  Note that by \eqref{canonical} we have
\[\vol(X)=\frac{(4\pi)^n}{n!(n+1)^n}\cdot L^n\]
and so we can restate the best known bounds in this context:
\begin{cor}\label{manycusps}Let $k$ be the number of cusps of $X$.  Then
\[\frac{\vol(X)}{k}\geq \begin{cases}
\frac{\pi^2}{2}&n=2\\
\frac{(4\pi)^n}{n(n+1)^n}&n=3,4,5\\
\frac{2^n}{n}&n\geq 6
\end{cases}\]
\end{cor}

The bound of Corollary \ref{manycusps} in dimension $n=2$ is sharp and due to Di Cerbo--Di Cerbo \cite{cerbosharp}.  For $n\geq 6$, the above bound is equal to that derived by Parker \cite{parker} in the case that the parabolic subgroups of $\Gamma$ are unipotent; we show that the same bound holds for the larger class of $\Gamma$ torsion-free at infinity.  On the other hand, the argument of Proposition \ref{cuspcount} could conceivably improve Parker's bound for all torsion-free $\Gamma$ if $m_\infty$ can be controlled sufficiently well.

Finally, we finish the proof of Corollary \ref{brody}:
\begin{cor}\label{brody2}
Let $\pi:X'\into X$ be a finite \'etale cover.  Then $\bar X'$ satisfies the Green-Griffiths conjecture with exceptional locus $D$ if $\pi$ ramifies to order at least 4 along every boundary component.
\end{cor}
\begin{proof}
By the argument of Corollary \ref{brody1}, we need $K_{\bar X'}+(1-\frac{n+1}{2})D'$ to be big.  By assumption, $\pi^*D\geq 4D'$, so
\begin{align*}
\pi^*\left((K_{\bar X}+D)-\frac{n+1}{2\pi}D\right)&= (K_{\bar X'}+D')-\frac{n+1}{2\pi}\pi^*D\\
&\leq (K_{\bar X'}+D')-\frac{n+1}{2}D'
 \end{align*}
The left hand side is big by Proposition \ref{nef}, so the right hand side is as well.
\end{proof}

The following Corollary is a well-known consequence of Nadel's theorem for arithmetic quotients but in fact the same proof holds for non-arithmetic quotients given the work of \cite{mok}.  We include it for completeness, but the main point of Corollary \ref{brody} is the improved control over the ramification order.

\begin{cor}
Every complex hyperbolic orbifold $X$ admits a finite \'etale cover $X'$ such that the toroidal compactification $\bar X'$ satisfies the Green--Griffiths conjecture with the boundary as exceptional locus.
\end{cor}

Of course, this is equivalent to the Baily--Borel compactification $X'^*$ having no nontrivial entire maps $\C\into X'^*$.

\bibliography{biblio}
\bibliographystyle{alpha}
\end{document}